\documentclass[10pt,draft,reqno]{amsart}
     \makeatletter
     \def\section{\@startsection{section}{1}%
     \z@{.7\linespacing\@plus\linespacing}{.5\linespacing}%
     {\bfseries
     \centering
     }}
     \def\@secnumfont{\bfseries}
     \makeatother
\newtheorem{theorem}{Theorem}[section]
\newtheorem{lemma}[theorem]{Lemma}
\newtheorem{proposition}[theorem]{Proposition}

\theoremstyle{definition}

\theoremstyle{remark}
\newtheorem{remark}[theorem]{Remark}
\numberwithin{equation}{section}
\begin{document}
\title{Exponential ergodicity of the jump-diffusion CIR process}
\author[P. Jin]{Peng Jin}
\address{Peng Jin: Fachbereich C, Bergische Universit\"at Wuppertal, 42119 Wuppertal, Germany}
\email{jin@uni-wuppertal.de}

\author[B. R\"udiger]{Barbara R\"udiger}
\address{Barbara R\"udiger: Fachbereich C, Bergische Universit\"at Wuppertal, 42119 Wuppertal, Germany}
\email{ruediger@uni-wuppertal.de}

\author[C. Trabelsi]{Chiraz Trabelsi}
\address{Chiraz Trabelsi: Department of Mathematics, University of Tunis El-Manar, 1060 Tunis, Tunisia}
\email{trabelsichiraz@hotmail.fr}

\begin{abstract}
In this paper we study the jump-diffusion CIR process (shorted as JCIR), which is an extension of the classical CIR model. The jumps of the JCIR are introduced with the help of a pure-jump L\'evy process $(J_t, t \ge 0)$. Under some suitable conditions on the L\'evy measure of $(J_t, t \ge 0)$, we derive a lower bound for the transition densities of the JCIR process. We also find some sufficient condition guaranteeing the existence of a Forster-Lyapunov function for the JCIR process, which allows us to prove its exponential ergodicity.
\end{abstract}
\maketitle

\section{Introduction}
The Cox-Ingersoll-Ross model (or CIR model) was introduced in \cite{MR785475} by John C. Cox, Jonathan E. Ingersoll and Stephen A. Ross
in order  to describe the random evolution of interest rates. The CIR model captures many features of the real world interest rates. In particular, the interest rate in  the CIR model is non-negative and mean-reverting. Because of its vast applications  in mathematical finance, some extenssions of the CIR model have been introduced and studied, see e.g. \cite{Duffie01,MR1850789,2013arXiv1301.3243L}.

In this paper we study an extension of the CIR model including jumps, the so-called jump-diffusion CIR process (shorted as JCIR). The JCIR process is defined as the unique strong solution $X:=(X_t, t\geq 0)$ to the following stochastic differential equation
\begin{equation}\label{jcir}
     dX_{t}=a(\theta-X_{t})dt +\sigma \sqrt{X_{t}}dW_{t}+dJ_t, \quad X_0\geq 0,
 \end{equation}
where  $a,\sigma>0$, $\theta \ge 0$ are constants,  $(W_{t},t\geq 0)$ is a $1$-dimensional Brownian motion and $(J_t,t\geq 0)$ is an independent pure-jump L\'evy process with its L\'evy measure $\nu$ concentrating on $(0,\infty)$ and satisfying
\begin{equation}\label{levy measure}
\int_{(0,\infty)}(\xi \wedge 1) \nu(d\xi) < \infty.
\end{equation}
The initial value $X_0$ is assumed to be independent of $(W_{t},t\geq 0)$ and $(J_t,t\geq 0)$.
We assume that all the above processes are defined on some filtered probability space $(\Omega, \mathcal{F}, (\mathcal{F})_{t \ge 0}, P)$. We remark that the existence and uniqueness of strong solutions to (\ref{jcir}) is guaranteed by the main results of \cite{MR2584896}.

The term $a(\theta-X_{t})$ in (\ref{jcir}) defines a mean reverting drift pulling the process towards its long-term value $\theta$ with a speed of adjustment equal to $a$. Since the diffusion coefficient in the SDE (\ref{jcir}) is degenerate at $0$ and only positive jumps are allowed, the JCIR process $(X_t, t\geq 0)$ stays non-negative if $X_0\geq 0$. This fact can be shown rigorously with the help of comparison theorems for SDEs,  for more details we refer the readers to \cite{MR2584896}.

The JCIR defined in (\ref{jcir}) includes the basic affine jump-diffusion (or BAJD) as a special case, in which the L\'evy process $(J_t,t\geq 0)$ takes the form of a compound poisson process with exponentially distributed jumps. The BAJD was introduced by Duffie and G\^{a}rleanu \cite{Duffie01} to describe the dynamics of default intensity.  It was also used in \cite{MR1850789} and \cite{MR2390186} as a short-rate model.  Motivated by some applications in finance, the long-time behavior of the BAJD has been well studied. According to the main results of  \cite{MR2779872,MR2390186}, the BAJD possesses a unique invariant probability measure, whose distributional properties was later investigated in \cite{PRT,MR2922631}. We remark that the results in \cite{MR2779872,MR2922631} are very general and hold for a large class of affine process with state space $\mathbb{R}_+$. The existence and some approximations of the transition densities of the BAJD can be found in \cite{MR3084047}. A closed formula of the transition densities of the BAJD was recently derived in \cite{PRT}.

In this paper we are interested in two problems concerning the JCIR defined in (\ref{jcir}). The first one is the transition density estimates of the JCIR. Our first main result of this paper is the derivation of a lower bound on the transition densities of the JCIR. Our idea to establish the lower bound of the transition densities is as follows. Like the BAJD, the JCIR is also an affine processes in $\mathbb{R}_+$. Roughly speaking, affine processes are Markov processes for which the logarithm of the characteristic function of the process is affine with respect to the initial state.  Affine processes on the canonical state space $\mathbb{R}^m_+\times \mathbb{R}^{n}$ have been investigated in \cite{MR1994043,MR1850789,MR2851694,MR3040553}. Based on the exponential-affine structure of the JCIR, we are able to compute its characteristic function explicitly. This enables us further to represent the distribution of the JCIR as the convolution of two distributions. The first distribution is known and coincides with the distribution of the CIR model. However, the second distribution is more complicated. We will give some sufficient condition such that the second distribution is singular at the point $0$. In this way we derive a lower bound estimate of the transition densities of the JCIR.

The other problem we consider here is the exponential ergodicity of the JCIR. According to the main results of \cite{MR2779872} (see also \cite{MR2390186}), the JCIR has a unique invariant probability measure $\pi$, given that some integrability condition on the l\'evy measure of $(J_t,t\geq 0)$ is satisfied. Under some sharper assumptions we show in this paper that the convergence of the law of the JCIR process to its invariant probability measure under the total variation norm is exponentially fast, which is called the exponential ergodicity.  Our method is the same as in \cite{PRT}, namely we show the existence of a Forster-Lyapunov function and then apply the general framework of \cite{MR1174380,MR1234294,MR1234295} to get the exponential ergodicity.

The remainder of this paper is organized as follows. In Section 2 we collect some key facts on the JCIR and in particular derive its characteristic function. In Section 3 we study the characteristic function of the JCIR and prove a lower bound of its transition densities. In Section 4 we show the  existence of a Forster-Lyapunov function and the exponential ergodicity for the JCIR.

\section{Preliminaries}
In this section we use the exponential-affine structure of the JCIR process to derive its characteristic functions.

We recall that the JCIR process $(X_t, t\geq 0)$ is defined to be the solution to (\ref{jcir}) and it  depends obviously on its initial value $X_0$. From now on we denote by $(X^x_t, t\geq 0)$ the JCIR process started from a initial point $x \ge 0$, namely
\begin{equation}\label{jcirx}
     dX^x_{t}=a(\theta-X^x_{t})dt +\sigma \sqrt{X^x_{t}}dW_{t}+dJ_t, \quad X^x_0=x.
 \end{equation}

Since the JCIR process is an affine process, the corresponding characteristic functions of $(X^x_t, t\geq 0)$ is of exponential-affine form:
\begin{equation}\label{characteristic function}
E\big[e^{uX^x_t}\big]=e^{\phi(t,u)+x\psi(t,u)}, \quad u \in \mathcal{U}:=\{u \in \mathbb{C}: \Re u \le 0\},
\end{equation}
where the functions $\phi(t,u)$ and $\psi(t,u)$ in turn are given as solutions of the generalized Riccati equations
\begin{equation}\label{eqRiccati}
\begin{cases} \partial_t \phi(t,u)=F\big(\psi(t,u)\big), \quad & \phi(0,u)=0, \\
 \partial_t \psi(t,u)=R\big(\psi(t,u)\big), \quad & \psi(0,u)=u \in \mathcal{U}, \\
\end{cases}
\end{equation}
with the functions $F$ and $R$ given by
\begin{align*}\label{functions F&R}
F(u)=&a\theta u+\int_{(0,\infty)}(e^{u\xi}-1)\nu(d\xi),  \\
R(u)=&\frac{\sigma^2u^2}{2}-au.
\end{align*}
Solving the system (\ref{eqRiccati}) give $\phi(t,u)$ and $\psi(t,u)$ in explicit form:
\begin{equation}\label{psi}
\psi(t,u)=\frac{ue^{-at}}{1-\frac{\sigma^2}{2a}u(1-e^{-at})}
\end{equation}
and
\begin{equation}\label{phi}
\phi(t,u)=-\frac{2a\theta}{\sigma^2}\log \big(1-\frac{\sigma^2}{2a}u(1-e^{-at})\big)
+ \int_0^t \int_{(0,\infty)}\Big(e^{\xi\psi(s,u)}-1\Big)\nu(d\xi)ds.
\end{equation}
Here the complex-valued logarithmic function $\log(\cdot)$ is understood to be its main branch defined on $\mathbb{C}-\{0\}$.
According to (\ref{characteristic function}),  (\ref{psi}) and (\ref{phi}) the characteristic functions of $(X^x_t, t\geq 0)$ is given by
\begin{align}\label{eqcharacteristic}
E[e^{uX^x_t}]=&  \big(1-\frac{\sigma^2}{2a}u(1-e^{-at})\big)^{-\frac{2a\theta}{\sigma^2}} \cdot \exp \Big (\frac{xue^{-at}}{1-\frac{\sigma^2}{2a}u(1-e^{-at})} \Big) \nonumber \\
& \qquad \cdot \exp \Big ( \int_0^t \int_0^{\infty}\Big(e^{\xi\psi(s,u)}-1\Big)\nu(d\xi)ds\Big)
\end{align}
for any $u \in \mathcal{U}$. Here the complex-valued power function $z^{-\frac{2a\theta}{\sigma^2}}:=\exp(-\frac{2a\theta}{\sigma^2} \log z)$ is also understood to be its main branch defined on $\mathbb{C}-\{0\}$.

We denote by
\begin{align}\label{IandII}
I:=&  \big(1-\frac{\sigma^2}{2a}u(1-e^{-at})\big)^{-\frac{2a\theta}{\sigma^2}} \cdot \exp \Big (\frac{xue^{-at}}{1-\frac{\sigma^2}{2a}u(1-e^{-at})} \Big)  \nonumber \\
II:=&  \exp \Big ( \int_0^t \int_0^{\infty}\Big(e^{\xi\psi(s,u)}-1\Big)\nu(d\xi)ds\Big)
\end{align}
the first and the second line of the characteristic function (\ref{eqcharacteristic}) respectively.\\


According to the parameters JCIR process we look at two special cases:

\subsection{Special case i): $\nu = 0$, no jumps}

 Notice that $\nu =0$ just yields the classical CIR model $(Y_t, t\geq 0)$ satisfying the following stochastic differential equation
 \begin{equation}\label{cir}
     dY^x_{t}=a(\theta-Y^x_{t})dt +\sigma \sqrt{Y^x_{t}}dW_{t}, \quad Y^x_0=x\geq 0.
 \end{equation}
It follows from (\ref{eqcharacteristic}) that the characteristic function of $(Y^x_t, t\geq 0)$ coincides with $I$, namely for $u \in \mathcal{U}$
 \begin{align}\label{eqcharaCIR}
E[e^{uY^x_t}]=&  \big(1-\frac{\sigma^2}{2a}u(1-e^{-at})\big)^{-\frac{2a\theta}{\sigma^2}} \cdot \exp \Big(\frac{xue^{-at}}{1-\frac{\sigma^2}{2a}u(1-e^{-at})} \Big) \\
=& I \nonumber
\end{align}
It is well known that the classical CIR model $(Y^x_t, t\geq 0)$ has transition density functions $f(t,x,y)$ given by
\begin{equation}\label{cirdensity}
f(t,x,y) =\kappa e^{-u-v}\Big ( \frac{v}{u} \Big)^{\frac{q}{2}} I_{q}\big (2(uv)^{\frac{1}{2}} \big)
\end{equation}
for $t>0, x>0$ and $y\ge0$, where
\begin{align*}
\kappa \equiv & \frac{2a}{\sigma^{2}\Big (1-e^{-at}\Big)}, &u \equiv \kappa xe^{-at},\\
v \equiv & \kappa y, &q \equiv  \frac{2a\theta}{\sigma^{2}}-1,
\end{align*}
and $I_{q}(\cdot)$ is the modified Bessel function of the first kind of order $q$. For $x=0$ the formula of the density function $p(t,x,y)$ is given by
\begin{equation}\label{cirdensity0}
p(t,0,y)= \frac{c}{\Gamma (q+1)}v^qe^{-v}
\end{equation}
for $t>0$ and $y \ge 0$.

\subsection{Special case ii): $\theta$=0 and $x=0$}

We denote by $(Z_t, t\geq 0)$ the JCIR process given by
\[
 dZ_{t}=-aZ_{t}dt +\sigma \sqrt{Z_{t}}dW_{t}+dJ_t, \quad Z_0=0.\]

In this particular case the characteristic functions of $(Z_t, t\geq 0)$ is equal to $II$, namely for $u \in \mathcal{U}$
\begin{equation*}
E[e^{uZ_t}]= \exp \Big ( \int_0^t \int_{(0,\infty)}\Big(e^{\xi\psi(s,u)}-1\Big)\nu(d\xi)ds\Big)
=II,
\end{equation*}
where
\begin{equation*}
\psi(t,u)=\frac{ue^{-at}}{1-\frac{\sigma^2}{2a}u(1-e^{-at})}.
\end{equation*}
One can notice that $II$ resembles the characteristic function of a compound poisson distribution.


\section{A lower bound for the transition densities of JCIR}
In this section we will find some conditions on the L\'evy measure $\nu$ of $(J_t,  t\ge0)$ such that an explicit lower bound for the transition densities of the JCIR process given in $(\ref{jcirx})$ can be derived. As a first step we note that the law of $X^x_t, t >0$, in $(\ref{jcirx})$ is absolutely continuous with respect to the Lebesgue measure and thus possesses a density function.

\begin{lemma}\label{edensity}
Consider the JCIR process $(X^x_t, t\geq 0)$ (started from $x \ge 0$) that is defined in $(\ref{jcirx})$. Then for any $t>0$ and $x\ge0$ the law of $X^x_t$ is absolutely continuous with respect to the Lebesgue measure and thus possesses a density function $p(t,x,y), \ y\ge0$.
\end{lemma}
\begin{proof}
As shown in the previous section, it holds
\begin{equation*}
    E[e^{uX^x_t}]=I \cdot II=E[e^{uY^x_t}]\cdot E[e^{uZ_t}],
\end{equation*}
therefore the law of $X^x_t$, denoted by $\mu_{X^x_t}$, is the convolution of the laws of $Y^x_t$ and $Z_t$. Since $(Y^x_t, t\geq 0)$ is the well-known CIR process and has transition density functions $f(t,x,y), t>0, x,y\ge0$ with respect to the Lebesgue measure, thus $\mu_{X^x_t}$ is also absolutely continuous with respect to the Lebesgue measure and possesses a density function.

\end{proof}

In order to get a lower bound for the transition densities of the JCIR process we need the following lemma.
\begin{lemma}\label{theorem1}
Suppose that $\int_{(0,1)} \xi \ln (\frac{1}{\xi})  \nu(d\xi)<\infty$. Then II defined in (\ref{IandII}) is the characteristic function of a compound Poisson distribution. In particular, $P(Z_t=0)>0$ for all $t>0$.
\end{lemma}

\begin{proof}
Since
\begin{align*}
E[e^{uZ_t}]= \exp \Big ( \int_0^t \int_{(0,\infty)}\Big(e^{\xi\psi(s,u)}-1\Big)\nu(d\xi)ds\Big)\\
= \exp \Big ( \int_0^t \int_{(0,\infty)}\displaystyle{\Big(e^{\frac{\xi ue^{-as}}{1-\frac{\sigma^2}{2a}.(1-e^{-as})u}}-1\Big)\nu(d\xi)ds\Big)},
\end{align*}
where $u \in \mathcal{U}$, we set
$$\Delta:= \int_0^t \int_{(0,\infty)}\Big(e^{\frac{\xi ue^{-as}}{1-\frac{\sigma^2}{2a}.(1-e^{-as})u}}-1\Big)\nu(d\xi)ds.$$
If we rewrite
\begin{equation}\label{eqess}
  \exp \Big(\frac{\xi e^{-as}u}{1-\frac{\sigma^2}{2a}\cdot (1-e^{-as})u} \Big) =  \exp \Big( \frac{\alpha u}{\beta -u} \Big ), \end{equation}
where
\begin{equation}\label{parameters}
\begin{cases} &\alpha  := \displaystyle{ \frac{2a\xi}{\sigma^2(e^{as}-1)} }>0\\  \\
&\beta : =\displaystyle{\frac{2ae^{as}}{\sigma^2( e^{as}-1)}}>0,
\end{cases}
\end{equation}
then we recognize that the right-hand side of (\ref{eqess}) is the characteristic function of a Bessel distribution with parameters $\alpha$ and $\beta$. Recall that a probability measure $\mu_{\alpha,\beta}$ on $\big (\mathbb{R}_+,\mathcal{B}(\mathbb{R}_+)\big)$ is called a Bessel distribution  with parameters $\alpha$ and $\beta$ if
\begin{equation}\label{besseldistru}
\mu_{\alpha,\beta}(dx)=e^{-\alpha}\delta_{0}(dx)+\beta e^{-\alpha-\beta x}\sqrt{\frac{\alpha}{\beta x}} \cdot I_1 (2\sqrt{\alpha \beta x})dx,
\end{equation}
where $\delta_0$ is the Dirac measure at the origin and $I_1$ is the modified Bessel function of the first kind, namely
\[
 I_1(r)=\frac{r}{2}\sum_{k=0}^{\infty}\frac{\big(\frac{1}{4}r^2\big)^k}{k ! (k+1)!}, \qquad r \in \mathbb{R}.
\]
For more properties of Bessel distributions we refer the readers to section $3$ of \cite{PRT}.  Let $\displaystyle{\hat{\mu}_{\alpha,\beta}}$ denote the characteristic function of the Bessel distribution $\displaystyle{\mu_{\alpha,\beta}}$ with parameters $\alpha$ and $\beta$ which are defined in (\ref{parameters}). It follows from \cite[Lemma 3.1]{PRT} that
\[
\displaystyle{\hat{\mu}_{\alpha,\beta}(u)}=\exp \Big( \frac{\alpha u}{\beta -u} \Big )=\exp \Big(\frac{\xi e^{-as}u}{1-\frac{\sigma^2}{2a}\cdot (1-e^{-as})u} \Big).
\]
Therefore
\begin{eqnarray*}
  \Delta &=& \int_0^t \int_{(0,\infty)} \big(\hat{\mu}_{\alpha,\beta}(u)-1\big)\nu(d\xi)ds \\
   &=&  \int_0^t \int_{(0,\infty)}\Big(e^{\frac{\alpha u}{\beta-u}}-e^{-\alpha}+e^{-\alpha}-1\Big)\nu(d\xi)ds.
\end{eqnarray*}
Set
\begin{eqnarray}\label{lambda}
  \lambda & :=&   \int_0^t \int_{(0,\infty)} \big(1-e^{-\alpha}\big)  \nu(d\xi)ds \nonumber \\
   &=&  \int_0^t \int_{(0,\infty)}  \Big(1-e^{-\frac{2a\xi}{\sigma^2(e^{as}-1)}}\Big)   \nu(d\xi)ds.
\end{eqnarray}
If $\lambda < \infty$, then
\begin{eqnarray*}
  \Delta &=& \int_0^t \int_{(0,\infty)}  \Big(e^{\frac{\alpha u}{\beta-u}}-e^{-\alpha}\Big)  \nu(d\xi)ds-\lambda \\
   &=& \lambda \Big(\frac{1}{\lambda}\int_0^t \int_{(0,\infty)}  \big(e^{\frac{\alpha u}{\beta-u}}-e^{-\alpha}\big)  \nu(d\xi)ds-1\Big).
\end{eqnarray*}
The fact that $\lambda<\infty$ will be shown later in this proof. Next we show that the term $\displaystyle{\frac{1}{\lambda}\int_0^t \int_{(0,\infty)}  \big(e^{\frac{\alpha u}{\beta-u}}-e^{-\alpha}\big)  \nu(d\xi)ds}$ can be viewed as the characteristic function of a probability measure $\rho$, which we now construct as mixtures of the following measures
\begin{equation*}
    m_{\alpha,\beta}(dx):=\beta e^{-\alpha-\beta x}\sqrt{\frac{\alpha}{\beta x}} \cdot I_1 (2\sqrt{\alpha \beta x})dx, \ x\geq 0,
\end{equation*}
where $I_1$ is the modified Bessel function of the first kind. Noticing that the measure $m_{\alpha,\beta}$ is the absolute continuous component of the measure $\mu_{\alpha,\beta}$ in (\ref{besseldistru}), we easily get
 \[
 \displaystyle{\hat{m}_{\alpha,\beta}(u)=\hat{\mu}_{\alpha,\beta}(u)-e^{-\alpha}=e^{\frac{\alpha u}{\beta-u}}-e^{-\alpha}},
 \]
 where $\hat{m}_{\alpha,\beta}(u):= \int_0^{\infty}e^{ux}m_{\alpha, \beta}(dx)$ for $u\in \mathcal{U}$. Since $\alpha$ and $\beta$ are functions in $\xi$ and $s$, we can define a measure $\rho$ on $\mathbb{R}_+$ as follows:
\begin{equation*}
    \rho := \frac{1}{\lambda}\int_0^t \int_{(0,\infty)}   m_{\alpha,\beta}  \  \nu(d\xi)ds.
\end{equation*}
By the definition of the constant $\lambda$ in (\ref{lambda}) we get
\begin{eqnarray*}
  \rho(\mathbb{R}_+) &=& \frac{1}{\lambda}\int_0^t \int_{(0,\infty)}   m_{\alpha,\beta}(\mathbb{R}_+) \nu(d\xi)ds \\
   &=& \frac{1}{\lambda}\int_0^t \int_{(0,\infty)}   (1-e^{-\alpha}) \nu(d\xi)ds \\
   &=& 1
\end{eqnarray*}
i.e. $\rho$ is a probability measure on $\mathbb{R}_+$ and for $u \in \mathcal{U}$
\begin{eqnarray*}
  \hat{\rho}(u) &=& \int_{(0,\infty)}e^{ux}\rho(dx) \\
   &=&  \frac{1}{\lambda}\int_0^t \int_{(0,\infty)}   (e^{\frac{\alpha u}{\beta-u}}-e^{-\alpha}) \nu(d\xi)ds.
\end{eqnarray*}
thus $\displaystyle{\Delta=\lambda(\hat{\rho}(u)-1)}$ and \ $\displaystyle{E[e^{uZ_t}]=e^{\lambda(\hat{\rho}(u)-1)}}$ is the characteristic function of a compound Poisson distribution. \\
Now we verify that $\lambda < \infty$. Since
\begin{eqnarray*}
 \lambda  &=&   \int_0^t \int_{(0,\infty)} \big(1-e^{-\alpha}\big)  \nu(d\xi)ds \\
   &=&  \int_0^t \int_{(0,\infty)}  \Big(1-e^{-\frac{2a\xi}{\sigma^2(e^{as}-1)}}\Big)   \nu(d\xi)ds \\
   &=&  \int_{(0,\infty)} \int_0^t \Big(1-e^{-\frac{2a\xi}{\sigma^2(e^{as}-1)}}\Big)   ds\nu(d\xi)
\end{eqnarray*}
We introduce the change of variables \ $\displaystyle{\frac{2a\xi}{\sigma^2(e^{as}-1)}}:=y$ \ and then get
\begin{eqnarray*}
  dy &=& - \frac{2a\xi}{\sigma^2(e^{as}-1)^2} \cdot ae^{as} ds\\
   &=& -y^2 \frac{\sigma^2}{2\xi}(\frac{2a\xi}{\sigma^2y}+1)ds.
\end{eqnarray*}
Therefore
\begin{eqnarray*}
  \lambda &=& \int_{(0,\infty)}  \nu(d\xi) \int_{\infty}^{\frac{2a\xi}{\sigma^2(e^{at}-1)}}(1-e^{-y})\frac{-2\xi}{2a\xi y+\sigma^2y^2} dy\\
   &=& \int_{(0,\infty)}  \nu(d\xi) \int^{\infty}_{\frac{2a\xi}{\sigma^2(e^{at}-1)}}(1-e^{-y})\frac{2\xi}{2a\xi y+\sigma^2y^2}dy\\
   &=&  \int_{(0,\infty)}  \nu(d\xi) \int^{\infty}_{\frac{\xi}{\delta}}(1-e^{-y})\frac{2\xi}{2a\xi y+\sigma^2y^2}dy,
\end{eqnarray*}
where $\displaystyle{\delta:=\frac{\sigma^2(e^{at}-1)}{2a}}$. Define
\begin{equation*}
    M(\xi):=\int^{\infty}_{\frac{\xi}{\delta}}(1-e^{-y})\frac{2\xi}{2a\xi y+\sigma^2y^2}dy,
\end{equation*}
 then $M(\xi)$ is continuous on $(0,\infty)$ and as $\xi\rightarrow 0$ we get
 \begin{eqnarray*}
   M(\xi)&=&\int^{1}_{\frac{\xi}{\delta}}(1-e^{-y})\frac{2\xi}{2a\xi y+\sigma^2y^2}dy+ 2\xi\int_1^\infty(1-e^{-y})\frac{dy}{2a\xi y+\sigma^2y^2}\\
    &\leq &  2\xi \int^{1}_{\frac{\xi}{\delta}}\frac{y}{2a\xi y+\sigma^2y^2}dy+2\xi\int_1^\infty\frac{1}{2a\xi y+\sigma^2y^2}dy\\
    &\le&  2\xi\int^{1}_{\frac{\xi}{\delta}}\frac{1}{2a\xi +\sigma^2y}dy+2\xi\int_1^\infty\frac{1}{2a\xi y+\sigma^2y^2}dy.
 \end{eqnarray*}
For $y\in[\frac{\xi}{\delta},1]$ we have
\begin{eqnarray*}
  2\xi\int^{1}_{\frac{\xi}{\delta}}\frac{1}{2a\xi +\sigma^2y}dy &=& \frac{2\xi}{\sigma^2}\Big[\ln(2a\xi+\sigma^2y)\Big]_{\frac{\xi}{\delta}}^{1}\\
   &=& \frac{2\xi}{\sigma^{2}}\ln(2a\xi+\frac{\sigma^2\xi}{\delta})-\ln(2a\xi+\sigma^2) \\
   &=&  c_1\xi+c_2\xi\ln(\frac{1}{\xi})\leq c_3\xi\ln(\frac{1}{\xi})
\end{eqnarray*}
and as $\xi\rightarrow\infty$,
\begin{eqnarray*}
  M(\xi) &\leq& \int^{\infty}_{\frac{\xi}{\delta}}(1-e^{-y})\frac{2\xi}{2a\xi y+\sigma^2y^2}dy \\
   &\leq&  \int^{\infty}_{\frac{\xi}{\delta}}\frac{2\xi}{2a\xi y+\sigma^2y^2}dy \leq 2\xi\int^{\infty}_{\frac{\xi}{\delta}}\frac{1}{\sigma^2y^2}  \\
   &=& 2\xi\frac{1}{\sigma^2}\int^{\infty}_{\frac{\xi}{\delta}}d(-\frac{1}{y}) = \frac{2\xi}{\sigma^2}\Big[-\frac{1}{y}\Big]_{\frac{\xi}{\delta}}^{\infty}\\
   &=& \frac{2\xi}{\sigma^2}\cdot \frac{\delta}{\xi}=\frac{2\delta}{\sigma^2}:=c_4<\infty.
\end{eqnarray*}
We yield
\begin{equation*}
    \lambda \leq c_3 \int_0^1\xi\ln(\frac{1}{\xi})\nu(d\xi)+c_4 \int^{\infty}_1 1 \nu(d\xi) < \infty .
\end{equation*}
\end{proof}

With the help of the previous lemma we can easily prove the following proposition.
\begin{proposition}\label{theo2}
Let $p(t,x,y), \ t>0, x, y\ge0$ denote the transition densities of the JCIR process $(X^x_t, t\geq 0)$ defined in $(\ref{jcirx})$. Suppose that $\int_{(0,1)} \xi \ln (\frac{1}{\xi}) \nu(d\xi)<\infty$, then
for all $t>0, x, y\ge0$ we have
\[
p(t,x,y) \ge C(t)  f(t,x,y),\]
where $C(t)>0$ for all $t>0$ and $f(t,x,y)$ are transition densities of the CIR process (without jumps).
\end{proposition}
\begin{proof}
According to Lemma (\ref{theorem1}), we have $P(Z_t=0)>0$. Define
\begin{equation*}
    C(t):=P(Z_t=0).
\end{equation*}
Since
\begin{equation*}
    E[e^{uX^x_t}]=I \cdot II=E[e^{uY^x_t}]\cdot E[e^{uZ_t}],
\end{equation*}
the law of $X^x_t$, denoted by $\mu_{X^x_t}$, is the convolution of the laws of $Y^x_t$ and $Z_t$.
Thus for all $ A \in \mathcal{B}(\mathbb{R}_+)$
\begin{eqnarray*}
  \mu_{X^x_t}(A) &=& \int_{\mathbb{R}_+}\mu_{Y^x_t}(A-y)\mu_{Z_t}(dy) \\
   &\geq& \int_{\{0\}}\mu_{Y^x_t}(A-y)\mu_{Z_t}(dy) \\
   &\ge& \mu_{Y^x_t}(A-\{0\})\mu_{Z_t}(\{0\})\\
 &\ge&  C(t) \cdot \mu_{Y^x_t}(A) \\
   &\ge& C(t)  \int_A f(t,x,y)dy,
\end{eqnarray*}
where $f(t,x,y)$ are the transition densities of the classical CIR process given in (\ref{cirdensity}). Since $ A \in \mathcal{B}(\mathbb{R}_+)$ is arbitrary, we yield
\[
p(t,x,y) \ge C(t)  f(t,x,y)\]
for all $t>0, x, y\ge0$.
\end{proof}


\section{Exponential ergodicity of JCIR}
In this section we find some sufficient conditions such that the JCIR process is exponentially ergodic. Our method relies on the lower bound of the transition densities we have derived in the previous section and the existence of a Forster-Lyapunov function.
\begin{lemma}\label{lyapunov}
Suppose that $\int_{(1,\infty)}\xi\nu(d\xi)<\infty.$ Then the function $V(x)=x$, $x\geq 0,$ is a Forster-Lyapunov function for the JCIR process defined in (\ref{jcirx}), in the sense that for all $t>0$, $x\geq 0$,
\[
E[V(X^x_t)] \le e^{-at}V(x)+M,
\]
where $0<M<\infty$ is a constant.
\end{lemma}

\begin{proof}
Since $\mu_{X^x_t}=\mu_{Y^x_t}\ast\mu_{Z_t}$, thus
\[
E[X^x_t]=E[Y^x_t]+E[Z_t].
\]
Since $(Y^x_t, t\geq 0)$ is the CIR process starting from $x$, it is known that $\mu_{Y^x_t}$ is a non-central Chi-squared distribution and thus $E[Y^x_t]<\infty$. Next we show that $E[Z_t]<\infty$.\\
Let $u\in(-\infty,0)$. Then by using Fatou's Lemma we get
\begin{eqnarray*}
  E[Z_t] &=& E[\displaystyle{\lim_{u\rightarrow 0}\frac{e^{uZ_t}-1}{u}}] \\
   &\leq& \displaystyle{\liminf_{u\rightarrow 0}}E[\frac{e^{uZ_t}-1}{u}] = \displaystyle{\liminf_{u\rightarrow 0}}\frac{E[e^{uZ_t}]-1}{u}
\end{eqnarray*}
Remember that
\begin{equation*}
    E[e^{uZ_t}]=II(u)= \exp \Big ( \int_0^t \int_{(0,\infty)}\Big(e^{\frac{\xi ue^{-as}}{1-\frac{\sigma^2}{2a}.(1-e^{-as})u}}-1\Big)\nu(d\xi)ds\Big)=e^{\Delta(u)},
\end{equation*}
then we have for all $u\leq 0$
\begin{eqnarray*}
  \frac{\partial}{\partial u}(\Big(e^{\frac{\xi ue^{-as}}{1-\frac{\sigma^2}{2a}.(1-e^{-as})u}}-1\Big)) &=& \frac{\xi e^{-as}(1-\frac{\sigma^2}{2a} \cdot (1-e^{-as})u)-\xi ue^{-as}(\frac{\sigma^2}{2a}\cdot(1-e^{-as}))}{(1-\frac{\sigma^2}{2a}\cdot (1-e^{-as})u)^2}\\
   & & \cdot  \ \displaystyle{e^{\frac{\xi ue^{-as}}{1-\frac{\sigma^2}{2a}.(1-e^{-as})u}}}\\
   &\leq& \frac{\xi e^{-as}}{(1-\frac{\sigma^2}{2a}.(1-e^{-as})u)^2} \leq \xi e^{-as}
\end{eqnarray*}
and further
\begin{equation*}
    \int_0^t\int_{(0,\infty)}\xi e^{-as}\nu(d\xi)ds <\infty.
\end{equation*}
Thus $\Delta(u)$ is differentiable in $u$ and
\begin{equation*}
    \Delta^\prime(0)=\int_0^t\int_{(0,\infty)}\xi e^{-as}\nu(d\xi)ds =\frac{1-e^{-at}}{a}\int_{(0,\infty)}\xi \nu(d\xi).
\end{equation*}
It follows that
\begin{eqnarray*}
     E[Z_t] &\leq& \displaystyle{\liminf_{u\rightarrow 0}}\frac{II(u)-II(0)}{u}  \\
   &=& II^\prime(0)= e^{\Delta(0)} \cdot \Delta^\prime(0)\\
   &=&  \frac{1-e^{-at}}{a}\int_{(0,\infty)}\xi \nu(d\xi).
\end{eqnarray*}
Therefore under the assumption $\displaystyle{\int_{(0,\infty)}\xi\nu(d\xi)<\infty}$ we have proved that $E[Z_t]<\infty$.
Furthermore
\begin{equation*}
    E[Z_t]=\frac{\partial}{\partial u}\Big(E[e^{uZ_t}]\Big)\Big|_{ u=0}=\frac{1-e^{-at}}{a}\int_{(0,\infty)}\xi \nu(d\xi).
\end{equation*}
On the other hand
 \begin{align*}
E[e^{uY^x_t}]=&  \big(1-\frac{\sigma^2}{2a}u(1-e^{-at})\big)^{-\frac{2a\theta}{\sigma^2}} \cdot \exp \Big(\frac{xue^{-at}}{1-\frac{\sigma^2}{2a}u(1-e^{-at})} \Big),
\end{align*}
with a similar argument as above we get
\begin{equation*}
    E[Y^x_t]=\frac{\partial}{\partial u}\Big(E[e^{uY^x_t}]\Big)\Big|_{ u=0}= \theta(1-e^{-at})+xe^{-at}.
\end{equation*}
Altogether we get
\begin{eqnarray*}
  E[X^x_t] &=& E[Y^x_t]+E[Z_t]\\
   &=& (1-e^{-at})\big(\theta+\frac{1-e^{-at}}{a}\big) + xe^{-at} \\
   &\leq& \theta + \frac{1}{a}+ xe^{-at} ,
\end{eqnarray*}
namely
\begin{equation*}
      E[V(X^x_t)]\leq \theta + \frac{1}{a}+ e^{-at}V(x).
\end{equation*}
\end{proof}

\begin{remark}
If $\int_{(1,\infty)}\xi\nu(d\xi)<\infty$, then it follows from \cite[Theorem 3.16]{MR2779872} and \cite[Proposition 3.1]{MR2390186} the existence and uniqueness of an invariant probability  measure $\pi$ for the JCIR process.
\end{remark}
Let $\|\cdot\|_{TV}$ denote the total-variation norm for signed measures on $\mathbb{R}_+$, namely
\[
\|\mu\|_{TV}=\sup_{A \in \mathcal{B}(\mathbb{R}_+)} \{|\mu(A)| \}.\]

Let $P^{t}(x,\cdot):=P(X^x_t\in \cdot)$ be the distribution of the JCIR process at time $t$ started from the  initial point $x \ge 0$.
Now we prove the main result of this paper.
\begin{theorem}
Assume that
\[
\int_{(1,\infty)} \xi  \ \nu(d\xi)<\infty \quad \mbox{and} \quad \int_{(0,1)} \xi \ln (\frac{1}{\xi}) \nu(d\xi)<\infty.
\]
Let $\pi$ be the unique invariant probability measure for the JCIR process. Then the JCIR process is exponentially ergodic, namely there exist constants $0<\beta <1 $ and $0<B<\infty$ such that
\begin{equation}\label{exerjcir}
\|P^{t}(x,\cdot) -\pi\|_{TV}\le B\big(x+1\big)\beta^t, \quad t \ge 0,  \quad x \in \mathbb{R}_{+}.
\end{equation}
\end{theorem}

\begin{proof}
Basically, we follow the proof of \cite[Theorem 6.1]{MR1234295}. For any $\delta>0$ we consider the $\delta$-skeleton chain $\eta^x_n:=X^x_{n\delta},  \ n \in \mathbb{Z}_+$. Then $(\eta^x_n)_{ n \in \mathbb{Z}_+}$ is a Markov chain on the state space $\mathbb{R}_+$ with transition kernel $P^{\delta}(x,\cdot)$ and starting point $\eta^x_0=x$. It is easy to see that the measure $\pi$ is also an invariant probability measure for the chain $(\eta^x_n)_{ n \in \mathbb{Z}_+}$, $x \ge 0$.

Let $V(x)=x$, $x \ge 0$.  It follows from  the Markov property and Lemma \ref{lyapunov} that
\[
E[V(\eta^x_{n+1})|\eta^x_0, \eta^x_1, \cdots, \eta^x_n] = \int_{\mathbb{R}_+}V(y)P^{\delta}(\eta^x_n,dy) \le e^{-a \delta }V(\eta^x_n)+M,
\]
where $M$ is a positive constant. If we set $V_0:=V$ and $V_n:=V(\eta^x_n)$, $n\in \mathbb{N}$, then
\[
E[V_1] \le  e^{-a \delta }V_0(x)+M
\]
and
\[
E[V_{n+1}|\eta^x_0, \eta^x_1, \cdots, \eta^x_n]  \le e^{-a \delta }V_n+M, \quad n\in \mathbb{N}.
\]

Now we proceed to show that the chain $(\eta^x_n)_{ n \in \mathbb{Z}_+}$, $x \ge 0$, is $\lambda$-irreducible, strong aperiodic, and all compact subsets of $\mathbb{R}_+$ are petite for the chain $(\eta^x_n)_{ n \in \mathbb{Z}_+}$.

``$\lambda$-irreducibility": We show that the Lebesgue measure $\lambda$ on $\mathbb{R}_+$ is an irreducibility measure for $(\eta^x_n)_{ n \in \mathbb{Z}_+}$. Let $A\in \mathcal{B}(\mathbb{R}_+)$ and $\lambda(A)>0$, then it follows from Proposition \ref{theo2} that
\[
P[\eta^x_1 \in A|\eta^x_0 =x]=P(X^x_{\delta}\in A)\ge C(\delta) \int_{A}f(\delta,x,y)dy>0,
\]
since $f(\delta,x,y)>0$ for any $x \in \mathbb{R}_+$ and $y>0$. This shows that the chain $(\eta^x_n)_{ n \in \mathbb{Z}_+}$ is irreducible with $\lambda$ being an irreducibility measure.

``Strong aperiodicity" (see \cite[p.561]{MR1174380} for a definition):  To show the strong aperiodicity of $(\eta^x_n)_{n \in \mathbb{Z}_0}$, we need to find a set $B \in \mathcal{B}(\mathbb{R}_+)$,
a probability measure $m$ with $m(B)=1$, and $\epsilon >0$ such that
\begin{equation}\label{ergo1}
L(x,B)>0, \qquad x \in \mathbb{R}_+
\end{equation}
and
\begin{equation}\label{ergo2}
P(\eta^x_1\in A) \ge \epsilon \cdot m(A), \quad x \in C, \quad A \in \mathcal{B}(\mathbb{R}_+),
\end{equation}
where $L(x,B):=P(\eta^x_n \in B \ \rm{for} \ \rm{some} \ \it{n} \in \mathbb{N})$. To this end set $B:=[0,1]$ and $g(y):=\inf_{x \in [0,1]} f(\delta,x,y)$, $y>0$. Since for fixed $y>0$ the function $f(\delta,x,y)$ strictly positive and continuous in $x \in [0,1]$, thus we have $g(y)>0$ and $0<\int_{(0,1]}g(y)dy\le 1$. Define
\[
m(A):=\frac{1}{\int_{(0,1]}g(y)dy}\int_{A \cap(0,1]}g(y)dy, \qquad A \in \mathcal{B}(\mathbb{R}_+).
\]
Then for any $x\in[0,1]$ and $A \in \mathcal{B}(\mathbb{R}_+)$ we get
\begin{align*}
P(\eta^x_1\in A)=&P(X^x_{\delta}\in A)\\
\ge  &  C(\delta) \int_{A}f(\delta,x,y)dy  \ge C(\delta) \int_{A \cap(0,1]}g(y)dy =C(\delta)m(A) \int_{(0,1]}g(y)dy,
\end{align*}
so (\ref{ergo2}) holds with $\epsilon:=C(\delta)\int_{(0,1]}g(y)dy$.

Obviously
\[
L(x,[0,1]) \ge P(\eta^x_1\in [0,1])= P(X^x_{\delta} \in [0,1])\ge C(\delta) \int_{[0,1]}f(\delta,x,y)dy>0
\]
for all $x \in \mathbb{R}_+$, which verifies (\ref{ergo1}).

``Compact subsets are petite": We have shown that $\lambda$ is an irreducibility measure for $(\eta^x_n)_{ n \in \mathbb{Z}_+}$. According to \cite[Theorem 3.4(ii)]{MR1174380}, to show that all compact sets are petit, it suffices to prove the Feller property of $(\eta^x_n)_{ n \in \mathbb{Z}_+}$, $x \ge 0$, but this follows from the fact that $(\eta^x_n)_{ n \in \mathbb{Z}_+}$ is a skeleton chain of the JCIR process, which is an affine process and possess the Feller property.

According to \cite[Theorem 6.3]{MR1174380} (see also the proof of \cite[Theorem 6.1]{MR1174380}), the probability measure $\pi$ is the only invariant probability measure of the chain $(\eta^x_n)_{ n \in \mathbb{Z}_+}$, $x \ge 0$,  and there exist constants $\beta \in (0,1) $ and $C\in(0,\infty)$ such that
\[
\|P^{\delta n}(x,\cdot) -\pi\|_{TV}\le C\big(x+1\big)\beta^n, \quad   n \in \mathbb{Z}_+,  \quad x \in \mathbb{R}_{+}
\]

Then for the rest of the proof we can proceed as in \cite[p.536]{MR1234295} and get the inequality (\ref{exerjcir}).
\end{proof}

\par\bigskip\noindent
{\bf Acknowledgements.} The research was supported by the Research Programm "DAAD - Transformation: Kurzma{\ss}nahmen 2012/13 Programm". This research was also carried out with the support of CAS - Centre for Advanced Study, at the Norwegian Academy of Science and Letter, research program SEFE.

\bibliographystyle{amsplain}

\begin{thebibliography}{99}

\bibitem{MR785475}
John~C. Cox, Jonathan~E. Ingersoll, Jr., and Stephen~A. Ross, \emph{A theory of
  the term structure of interest rates}, Econometrica \textbf{53} (1985),
  no.~2, 385--407.

\bibitem{MR1994043}
D.~Duffie, D.~Filipovi{\'c}, and W.~Schachermayer, \emph{Affine processes and
  applications in finance}, Ann. Appl. Probab. \textbf{13} (2003), no.~3,
  984--1053.

\bibitem{Duffie01}
Darrell Duffie and Nicolae G\^{a}rleanu, \emph{Risk and valuation of
  collateralized debt obligations}, Financial Analysts Journal \textbf{57}
  (2001), no.~1, pp. 41--59 (English).

\bibitem{MR1850789}
Damir Filipovi{\'c}, \emph{A general characterization of one factor affine term
  structure models}, Finance Stoch. \textbf{5} (2001), no.~3, 389--412.
  \MR{1850789 (2002f:91041)}

\bibitem{MR3084047}
Damir Filipovi{\'c}, Eberhard Mayerhofer, and Paul Schneider, \emph{Density
  approximations for multivariate affine jump-diffusion processes}, J.
  Econometrics \textbf{176} (2013), no.~2, 93--111. \MR{3084047}

\bibitem{MR2584896}
Zongfei Fu and Zenghu Li, \emph{Stochastic equations of non-negative processes
  with jumps}, Stochastic Process. Appl. \textbf{120} (2010), no.~3, 306--330.
  \MR{2584896 (2011d:60178)}

\bibitem{PRT}
Peng Jin, Barbara R{\"u}diger, and Chiraz Trabelsi, \emph{{Positive Harris
  recurrence and exponential ergodicity of the basic affine jump-diffusion}},
  ArXiv e-prints (2015).

\bibitem{MR2779872}
Martin Keller-Ressel, \emph{Moment explosions and long-term behavior of affine
  stochastic volatility models}, Math. Finance \textbf{21} (2011), no.~1,
  73--98. \MR{2779872 (2012e:91126)}

\bibitem{MR2922631}
Martin Keller-Ressel and Aleksandar Mijatovi{\'c}, \emph{On the limit
  distributions of continuous-state branching processes with immigration},
  Stochastic Process. Appl. \textbf{122} (2012), no.~6, 2329--2345.
  \MR{2922631}

\bibitem{MR2851694}
Martin Keller-Ressel, Walter Schachermayer, and Josef Teichmann, \emph{Affine
  processes are regular}, Probab. Theory Related Fields \textbf{151} (2011),
  no.~3-4, 591--611. \MR{2851694 (2012k:60219)}

\bibitem{MR3040553}
\bysame, \emph{Regularity of affine processes on general state spaces},
  Electron. J. Probab. \textbf{18} (2013), no. 43, 17. \MR{3040553}

\bibitem{MR2390186}
Martin Keller-Ressel and Thomas Steiner, \emph{Yield curve shapes and the
  asymptotic short rate distribution in affine one-factor models}, Finance
  Stoch. \textbf{12} (2008), no.~2, 149--172. \MR{2390186 (2009c:60222)}

\bibitem{2013arXiv1301.3243L}
Z.~{Li} and C.~{Ma}, \emph{{Asymptotic properties of estimators in a stable
  Cox-Ingersoll-Ross model}}, ArXiv e-prints (2013).

\bibitem{MR1174380}
Sean~P. Meyn and R.~L. Tweedie, \emph{Stability of {M}arkovian processes. {I}.
  {C}riteria for discrete-time chains}, Adv. in Appl. Probab. \textbf{24}
  (1992), no.~3, 542--574. \MR{1174380 (93g:60143)}

\bibitem{MR1234294}
\bysame, \emph{Stability of {M}arkovian processes. {II}. {C}ontinuous-time
  processes and sampled chains}, Adv. in Appl. Probab. \textbf{25} (1993),
  no.~3, 487--517. \MR{1234294 (94g:60136)}

\bibitem{MR1234295}
\bysame, \emph{Stability of {M}arkovian processes. {III}. {F}oster-{L}yapunov
  criteria for continuous-time processes}, Adv. in Appl. Probab. \textbf{25}
  (1993), no.~3, 518--548. \MR{1234295 (94g:60137)}

\end{thebibliography}

\end{document}